\documentclass[a4paper,11pt]{article}
\usepackage{amsmath,amsthm,amssymb,enumitem,xcolor}
\usepackage[hang]{footmisc}

\usepackage[top=27mm,left=27mm,right=27mm,bottom=27mm,footskip=15mm,footnotesep=10mm]{geometry}
\usepackage[nosort,nocompress,noadjust]{cite}

\usepackage[bookmarks=false,hyperfootnotes=false,colorlinks,linktoc=all,
    linkcolor={red!60!black},
    citecolor={blue!50!black},
    urlcolor={blue!80!black}]{hyperref}

\renewcommand{\eqref}[1]{\hyperref[#1]{(\ref{#1})}}

\newlist{enumlist}{enumerate}{2}
\setlist[enumlist,1]{labelindent=0cm,label=(\roman*),ref=(\roman*),labelwidth=4.5ex,labelsep=1ex,leftmargin=5.5ex,align=right,topsep=0.5ex,itemsep=1ex,parsep=1ex}
\setlist[enumlist,2]{labelindent=0cm,label=\theenumlisti.\arabic*.,ref=\arabic*,labelwidth=5ex,labelsep=0.5ex,leftmargin=5.5ex,align=left,topsep=0.5ex,itemsep=1ex,parsep=1ex}

\newlist{itemlist}{itemize}{1}
\setlist[itemlist]{labelindent=0cm,label=$\bullet$,labelwidth=2.5ex,labelsep=0.5ex,leftmargin=3ex,align=left,topsep=0.5ex,itemsep=1ex,parsep=1ex}

\numberwithin{equation}{section}

{\theoremstyle{definition}\newtheorem{definitiona}{Definition}[section]
\newtheorem{remarka}[definitiona]{Remark}
\newtheorem{examplea}[definitiona]{Example}}

\newtheorem{propositiona}[definitiona]{Proposition}
\newtheorem{lemmaa}[definitiona]{Lemma}
\newtheorem{theorema}[definitiona]{Theorem}
\newtheorem{corollarya}[definitiona]{Corollary}

\newtheorem{letterthma}{Theorem}
\renewcommand{\theletterthma}{\Alph{letterthma}}

{\theoremstyle{definition}
\newtheorem{letterdefa}[letterthma]{Definition}
\newtheorem{letterremarka}[letterthma]{Remark}}

\newenvironment{remark}[1][]{\begin{remarka}[#1]\setlist[enumlist,1]{label=(\roman*),ref=\theremarka(\roman*)}}{\end{remarka}}

\newenvironment{proposition}[1][]{\begin{propositiona}[#1]\setlist[enumlist,1]{label=(\roman*),ref=\thepropositiona(\roman*)}}{\end{propositiona}}
\newenvironment{lemma}[1][]{\begin{lemmaa}[#1]\setlist[enumlist,1]{label=(\roman*),ref=\thelemmaa(\roman*)}}{\end{lemmaa}}
\newenvironment{theorem}[1][]{\begin{theorema}[#1]\setlist[enumlist,1]{label=(\roman*),ref=\thetheorema(\roman*)}}{\end{theorema}}

\newenvironment{letterthm}[1][]{\begin{letterthma}[#1]\setlist[enumlist,1]{label=(\roman*),ref=\theletterthma.(\roman*)}}{\end{letterthma}}
\newenvironment{letterdef}[1][]{\begin{letterdefa}[#1]\setlist[enumlist,1]{label=(\roman*),ref=\theletterdefa.(\roman*)}}{\end{letterdefa}}
\newenvironment{letterremark}[1][]{\begin{letterremarka}[#1]\setlist[enumlist,1]{label=(\roman*),ref=\theletterremarka.(\roman*)}}{\end{letterremarka}}

\newcommand{\C}{\mathbb{C}}

\newcommand{\N}{\mathbb{N}}

\newcommand{\R}{\mathbb{R}}
\newcommand{\T}{\mathbb{T}}
\newcommand{\Z}{\mathbb{Z}}

\newcommand{\cG}{\mathcal{G}}
\newcommand{\cF}{\mathcal{F}}

\newcommand{\cO}{\mathcal{O}}

\newcommand{\cU}{\mathcal{U}}

\newcommand{\al}{\alpha}
\newcommand{\be}{\beta}
\newcommand{\eps}{\varepsilon}
\newcommand{\vphi}{\varphi}

\newcommand{\si}{\sigma}

\newcommand{\Ad}{\operatorname{Ad}}
\newcommand{\Aut}{\operatorname{Aut}}

\newcommand{\Inn}{\operatorname{Inn}}

\newcommand{\Out}{\operatorname{Out}}

\newcommand{\altil}{\widetilde{\alpha}}
\newcommand{\betil}{\widetilde{\beta}}

\newcommand{\alhat}{\widehat{\alpha}}
\newcommand{\behat}{\widehat{\beta}}
\newcommand{\Lambdahat}{\widehat{\Lambda}}

\newcommand{\ot}{\otimes}
\newcommand{\id}{\mathord{\text{\rm id}}}

\newcommand{\actson}{\curvearrowright}

\newcommand{\Nbar}{\overline{\N}}
\newcommand{\Autmp}{\operatorname{Aut}_{\text{\rm mp}}}
\newcommand{\Exp}{\operatorname{Exp}}

\begin{document}

\begin{center}
{\boldmath\LARGE\bf Every locally compact group is the outer\vspace{0.5ex}\\
automorphism group of a II$_1$ factor}

\vspace{1ex}

{\sc by Stefaan Vaes\footnote{KU~Leuven, Department of Mathematics, Leuven (Belgium), stefaan.vaes@kuleuven.be\\ Supported by FWO research project G090420N of the Research Foundation Flanders and by Methusalem grant METH/21/03 –- long term structural funding of the Flemish Government.}}
\end{center}

\begin{abstract}\noindent
We prove that every locally compact second countable group $G$ arises as the outer automorphism group $\Out M$ of a II$_1$ factor, which was so far only known for totally disconnected groups, compact groups and a few isolated examples. We obtain this result by proving that every locally compact second countable group is a \emph{centralizer group}, a class of Polish groups that arise naturally in ergodic theory and that may all be realized as $\Out M$.
\end{abstract}

\section{Introduction and main results}

The \emph{outer automorphism group} $\Out M$ of a II$_1$ factor $M$ is the quotient of the group $\Aut M$ of all $*$-automorphisms of $M$ by the normal subgroup $\Inn M$ of inner automorphisms, of the form $\Ad u$ with $u \in \cU(M)$. In \cite{IPP05} answering a question of Connes from the 1970s, it was proven that there exist II$_1$ factors $M$ with trivial outer automorphism group. This provided at the same time the first complete computation of the outer automorphism group of any II$_1$ factor.

A II$_1$ factor $M$ is called \emph{full} if $\Inn M$ is a closed subgroup of $\Aut M$. This is equivalent with $M$ not having property Gamma in the sense of Murray and von Neumann. If $M$ is full and has separable predual, $\Out M$ is a Polish group. It is a natural and by now well-studied question which Polish groups arise as the outer automorphism group of a full II$_1$ factor with separable predual. The following families of Polish groups could be realized in this way: compact abelian groups in \cite{IPP05}, countable discrete groups in \cite{PV06,Vae07}, arbitrary compact groups in \cite{FV07} and several families of non locally compact groups in \cite{PV21}, including all closed subgroups of the group of all permutations of a countably infinite set.

All this however left open the question if all locally compact second countable groups, in particular Lie groups, can be realized as $\Out M$. We solve this problem here and the following is our main result.

\begin{letterthm}\label{thm.main-every-lc}
Every locally compact second countable group $G$ arises as the outer automorphism group $\Out M$ of a full II$_1$ factor $M$ with separable predual.
\end{letterthm}

Denote by $(X,\mu)$ the standard nonatomic probability space. We consider the Polish group $\Aut(X,\mu)$ of all \emph{nonsingular} automorphisms of $(X,\mu)$, identifying two automorphisms that coincide almost everywhere. The \emph{measure preserving} automorphisms form a closed subgroup $\Autmp(X,\mu)$. In \cite[Theorem E]{Dep10}, it was proven that for every countable subgroup $\Lambda < \Autmp(X,\mu)$, the \emph{centralizer} $C(\Lambda)$ of $\Lambda$ inside $\Autmp(X,\mu)$ is a Polish group that arises as the outer automorphism group $\Out M$ of a full II$_1$ factor $M$ with separable predual. A more direct proof of this result was given in \cite[Proposition 8.5]{PV21}.

To prove Theorem \ref{thm.main-every-lc}, it thus suffices to realize every locally compact second countable group as a centralizer inside $\Autmp(X,\mu)$. Note here that a subgroup $G < \Autmp(X,\mu)$ is of the form $C(\Lambda)$ for a subgroup $\Lambda < \Autmp(X,\mu)$ if and only if $G$ equals its double centralizer $C(C(G))$.

\begin{letterdef}\label{def.centralizer-group}
We say that a Polish group $G$ is a \emph{centralizer group} if there exists an ergodic probability measure preserving action $\Lambda \actson (X,\mu)$ of a countable group $\Lambda$ on a standard probability space $(X,\mu)$ such that $G$ is isomorphic with the centralizer of $\Lambda$ inside $\Autmp(X,\mu)$.
\end{letterdef}

So as explained above, by \cite{Dep10,PV21} for every centralizer group $G$, there exists a full II$_1$ factor $M$ with separable predual such that $G \cong \Out M$. Since every locally compact group $G$ is the centralizer of $G$ acting on itself (by transformations that are not probability measure preserving), Theorem \ref{thm.main-every-lc} is thus a consequence of the following result, which also shows that the class of centralizer groups is quite natural.

\begin{letterthm}\label{thm.main-centralizer-group}
The following Polish groups are centralizer groups.
\begin{enumlist}
\item\label{thm.main-centralizer-group.i} The centralizer of any subgroup of the Polish group $\Autmp(X,\mu)$ of measure preserving automorphisms of a \emph{possibly infinite} $\si$-finite standard measure space $(X,\mu)$.
\item\label{thm.main-centralizer-group.ii} The centralizer of any subgroup of the Polish group $\Aut(X,\mu)$ of \emph{nonsingular} automorphisms of a standard probability space $(X,\mu)$.
\end{enumlist}
\end{letterthm}

To prove the first statement, we make use of \emph{Poisson suspensions}, while we deduce the second statement from the first by using the \emph{Maharam extension}. In Sections \ref{sec.permanence} and \ref{sec.closed-groups-of-permutations}, we prove a few permanence properties for the class of centralizer groups, we reprove that every closed subgroup of the group $S_\infty$ of all permutations of a countably infinite set is a centralizer group (cf.\ \cite[Corollary 8.2]{PV21}) and that $\Aut K$ is a centralizer group for all compact second countable groups $K$.

\begin{letterremark}
To summarize, we then get the following list of centralizer groups, which can all be realized as the outer automorphism group of a full II$_1$ factors with separable predual.
\begin{enumlist}
\item Locally compact second countable groups and their direct integrals (Theorem \ref{thm.main-every-lc} and Proposition \ref{prop.direct-integral}).
\item Closed subgroups of $S_\infty$, the group of all permutations of a countably infinite set (\cite[Corollary 8.2]{PV21} and Theorem \ref{thm.closed-subgroup-Sinfty}).
\item The automorphism group $\Aut K$ of a compact second countable group $K$ (Proposition \ref{prop.Aut-K}).
\item The unitary group $\cU(N)$ of a von Neumann algebra $N$ with separable predual (\cite[Corollary 8.2]{PV21}).
\item A direct product of a countable family of centralizer groups (Proposition \ref{prop.direct-product}).
\end{enumlist}
\end{letterremark}

\section{Preliminaries}

We start by recalling the concept of \emph{Poisson suspensions} and refer to \cite{JRD20} for details and further references. Let $(Y,\eta)$ be a standard $\si$-finite measure space. A measure on $Y$ of the form $\sum_{i \in I} \delta_{y_i}$, where $(y_i)_{i \in I}$ is a countable set of points in $Y$, is called a \emph{counting measure}. Write $\Nbar = \N \cup \{0,+\infty\}$. The set $X$ of all counting measures $\nu$ on $Y$ has a unique standard Borel structure such that the maps
$$N_A : X \to \Nbar : \nu \mapsto \nu(A)$$
are Borel for every Borel set $A \subset Y$. There is a unique probability measure $\mu$ on $X$ such that for every Borel set $A \subset Y$ with $0 < \eta(A) < +\infty$, the random variable $N_A$ has a Poisson distribution with intensity $\eta(A)$, i.e.
$$\mu\bigl(N_A^{-1}(\{k\})\bigr) = \exp(-\eta(A)) \, \frac{\eta(A)^k}{k!} \quad\text{for all $k \in \N \cup \{0\}$,}$$
and such that the random variables $N_{A_1},\ldots,N_{A_n}$ are mutually independent for all disjoint Borel sets $A_i \subset Y$ with $0 < \eta(A_i) < +\infty$. Also note that $N_A + N_B = N_{A \cup B}$ when $A,B \subset Y$ are disjoint Borel sets.

There then is a unique isometric real linear map
\begin{equation}\label{eq.map-theta}
\theta : L^2_\R(Y,\eta) \to L^2_\R(X,\mu) : \theta(1_A) = N_A - \eta(A) 1
\end{equation}
for all Borel sets $A \subset Y$ with $\eta(A) < +\infty$. One checks that
\begin{equation}\label{eq.char-function-theta}
E_\mu(\exp(it \theta(\xi))) = \exp\Bigl(\int_Y (\exp(i t \xi) - i t \xi - 1) \, d\eta \Bigr) \quad\text{for all $\xi \in L^2_\R(Y,\eta)$,}
\end{equation}
where the right hand side is well-defined because $|\exp(ir) - ir - 1| \leq r^2$ for all $r \in \R$.

Denote by $L^2_\R(Y,\eta)^{\ot_{\text{\rm symm}}^{n}}$ the closed subspace of symmetric vectors in $L^2_\R(Y,\eta)^{\ot^n}$. Consider the symmetric Fock space
$$\cF_s(L^2_\R(Y,\eta)) = \R 1 \oplus \bigoplus_{n=1}^\infty L^2_\R(Y,\eta)^{\ot_{\text{\rm symm}}^{n}}$$
and define for every $\xi \in L^2_\R(Y,\eta)$ the vector $\Exp(\xi) \in \cF_s(L^2_\R(Y,\eta))$ by
$$\Exp(\xi) = 1 \oplus \bigoplus_{n=1}^\infty \frac{1}{\sqrt{n!}} \xi^{\ot n} \; .$$
Note that $\langle \Exp(\xi),\Exp(\zeta)\rangle = \exp(\langle \xi,\zeta\rangle)$.

Then $\theta : L^2_\R(Y,\eta) \to L^2_\R(X,\mu)$ uniquely extends to an isometric, bijective, real linear map
\begin{equation}\label{eq.Fock-space-decomposition}
\Theta: \cF_s(L^2_\R(Y,\eta)) \to L^2_\R(X,\mu) : \Theta(\Exp(\exp(\xi) - 1)) = E_\mu(\exp(\theta(\xi)))^{-1} \, \exp(\theta(\xi))
\end{equation}
for all simple functions $\xi \in L^2_\R(Y,\eta)$, i.e.\ finite linear combinations of $1_A$ with $A \subset Y$ Borel and $\eta(A) < +\infty$.

To every $\al \in \Autmp(Y,\eta)$ corresponds the Poisson suspension $\alhat \in \Autmp(X,\mu)$ defined by $\alhat(\nu) = \al_*(\nu)$ for every counting measure $\nu$ on $Y$. Every measure preserving automorphism $\be$ of a $\sigma$-finite standard measure space $(Z,\zeta)$, also gives rise to the canonical orthogonal transformation $U_\be$ of $L^2_\R(Z,\zeta)$. We then get that
\begin{equation}\label{eq.suspension-respects-Fock}
U_{\alhat} \circ \Theta = \Theta \circ \Bigl( \id \oplus \bigoplus_{n=1}^\infty U_\al^{\ot_{\text{\rm symm}}^{n}}\Bigr) \quad\text{for all $\al \in \Autmp(Y,\eta)$.}
\end{equation}

\section{Centralizer groups, proof of Theorems \ref{thm.main-every-lc} and \ref{thm.main-centralizer-group}}\label{sec.centralizer-groups}

The following is the key technical lemma showing that in certain cases, the centralizer of a Poisson suspension equals the Poisson suspension of the centralizer.

\begin{lemma}\label{lem.reduction-to-base}
Let $(Y,\eta)$ be a standard $\sigma$-finite measure space and $\Lambda < \Autmp(Y,\eta)$ a subgroup. Denote by $(X,\mu)$ the Poisson suspension of $(Y,\eta)$ and consider $\Lambdahat < \Autmp(X,\mu)$ given by $\Lambdahat = \{\behat \mid \be \in \Lambda\}$. Assume that there exists a sequence $\be_n \in \Lambda$ such that the orthogonal transformations $U_{\be_n}$ of $L^2_\R(Y,\eta)$ satisfy $U_{\be_n} \to \rho \cdot \id$ weakly for some $0 < \rho < 1$.

Then $\Lambdahat < \Autmp(X,\mu)$ is ergodic and the map $\al \mapsto \alhat$ is a bijection between the centralizer of $\Lambda$ inside $\Autmp(Y,\eta)$ and the centralizer of $\Lambdahat$ inside $\Autmp(X,\mu)$.
\end{lemma}

\begin{proof}
Consider the symmetric Fock space $\cF_s(L^2_\R(Y,\eta))$ and the isometry $\Theta : \cF_s(L^2_\R(Y,\eta)) \to L^2_\R(X,\mu)$ given by \eqref{eq.Fock-space-decomposition}. By \eqref{eq.suspension-respects-Fock} and the assumption that $U_{\be_n} \to \rho \cdot \id$ weakly,
$$\Theta^* \circ U_{\widehat{\be_n}} \circ \Theta \to T := \id \oplus \bigoplus_{n=1}^\infty \rho^n \cdot \id \quad\text{weakly.}$$
So if $\xi \in L^2_\R(X,\mu)$ is $\Lambdahat$-invariant, it follows that $\Theta^*(\xi) \in \R 1$ and thus $\xi \in \R 1$, proving the ergodicity of $\Lambdahat < \Autmp(X,\mu)$.

Take $\vphi \in \Autmp(X,\mu)$ centralizing $\Lambdahat$. It follows that $\Theta^* \circ U_\vphi \circ \Theta$ commutes with $T$, and thus globally preserves the eigenspace $L^2_\R(Y,\eta)$, so that $U_\vphi$ globally preserves $\theta(L^2_\R(Y,\eta))$, with $\theta$ defined by \eqref{eq.map-theta}. It follows from \cite[Proposition 4.4]{Roy08} that $\vphi = \alhat$ for a uniquely defined $\al \in \Autmp(Y,\eta)$. For completeness, we provide the following easy argument to obtain this result.

Define the orthogonal transformation $U$ of $L^2_\R(Y,\eta)$ such that $U_\vphi \circ \theta = \theta \circ U$. Take an arbitrary $\xi \in L^2_\R(Y,\eta)$. The characteristic function of the random variable $\theta(\xi)$ is given by \eqref{eq.char-function-theta}. This is saying that $\theta(\xi)$ is infinitely divisible with L\'{e}vy measure $\rho_\xi$ given by $\rho_\xi(\cU) = \eta(\xi^{-1}(\cU))$ for all Borel sets $\cU \subset \R \setminus \{0\}$. Since $\vphi$ is measure preserving, $\theta(\xi)$ and $U_\vphi(\theta(\xi))$ have the same distribution and hence, the same L\'{e}vy measure. Since $U_\vphi(\theta(\xi)) = \theta(U(\xi))$, it follows that $\rho_\xi = \rho_{U(\xi)}$ for all $\xi \in L^2_\R(Y,\eta)$.

Since $\xi \geq 0$ a.e.\ if and only if $\rho_\xi(-\infty,0) = 0$ and since $\xi \leq 1$ a.e.\ if and only if $\rho_\xi(1,+\infty) = 0$, we conclude that the orthogonal transformation $U \in \cO(L^2_\R(Y,\eta))$ maps positive functions to positive functions and maps functions $0 \leq \xi \leq 1$ to functions $0 \leq U(\xi) \leq 1$. This is only possible if $U = V_\al$ for some $\al \in \Autmp(Y,\eta)$. So, $\vphi = \alhat$.

Since $\vphi$ commutes with $\Lambdahat$, we get that $\al$ commutes with $\Lambda$ and the lemma is proven.
\end{proof}

\begin{lemma}\label{lem.aut-to-half-weakly}
Let $Y$ be a standard Borel space and $\eta$ an infinite, $\si$-finite, nonatomic Borel measure on $Y$. There exists a sequence $\be_n \in \Autmp(Y,\eta)$ such that the orthogonal transformations $U_{\be_n}$ of $L^2_\R(Y,\eta)$ satisfy $U_{\be_n} \to (1/2) \cdot \id$ weakly.
\end{lemma}
\begin{proof}
Since all such $(Y,\eta)$ are isomorphic, we may realize $(Y,\eta) = (X \times \Z , \mu \times \lambda)$ with $(X,\mu) = (\{0,1\}^\N,\mu_0^\N)$, $\mu_0(0) = \mu_0(1) = 1/2$, and $\lambda$ the counting measure on $\Z$. We define $\beta_n \in \Autmp(Y,\eta)$ by
$$\beta_n(x,k) = \begin{cases} (x,k+n) &\;\;\text{if $x_n = 0$,}\\ (x,k) &\;\;\text{if $x_n = 1$.}\end{cases}$$
Take $N \in \N$, $A_0,A_1 \in \{0,1\}^N$ and $B_0,B_1 \subset [-N,N] \cap \Z$. Write $C_i = A_i \times \{0,1\}^{> N} \subset X$. When $n > 2N$, we have that
$$\langle U_{\beta_n}(1_{C_0} \ot 1_{B_0}) , 1_{C_1} \ot 1_{B_1} \rangle = \frac{1}{2} \langle 1_{C_0} \ot 1_{B_0} , 1_{C_1} \ot 1_{B_1} \rangle \; .$$
Since these vectors $1_C \ot 1_B$ are total in $L^2_\R(Y,\eta)$, it follows that $U_{\be_n} \to (1/2) \cdot \id$ weakly.
\end{proof}

\begin{proof}[{Proof of Theorem \ref{thm.main-centralizer-group}}]
To prove (i), take a standard $\si$-finite measure space $(Y_0,\eta_0)$ and a subgroup $\Lambda_0 < \Autmp(Y_0,\eta_0)$. Denote by $G$ the centralizer of $\Lambda_0$ inside $\Autmp(Y_0,\eta_0)$. We have to prove that $G$ is a centralizer group in the sense of Definition \ref{def.centralizer-group}.

Take an arbitrary nonatomic, infinite, $\si$-finite standard measure space $(Y_1,\eta_1)$. By Lemma \ref{lem.aut-to-half-weakly}, take a sequence $\be_n \in \Autmp(Y_1,\eta_1)$ such that $U_{\be_n} \to (1/2) \cdot \id$ weakly. Denote $(Y,\eta) = (Y_0 \times Y_1,\eta_0 \times \eta_1)$. Define $\Lambda_1 < \Autmp(Y,\eta)$ by $\Lambda_1 = \Lambda_0 \times \Autmp(Y_1,\eta_1)$.

The centralizer of $\Lambda_1$ inside $\Autmp(Y,\eta)$ equals $G \times \id$. The sequence $\id \times \be_n \in \Lambda$ satisfies $U_{\id \times \be_n} \to (1/2) \cdot \id$ weakly. Denote by $(X,\mu)$ the Poisson suspension of $(Y,\eta)$ and denote by $\widehat{\Lambda_1} < \Aut(X,\mu)$ the Poisson suspension $\Lambda_1$. By Lemma \ref{lem.reduction-to-base}, the subgroup $\widehat{\Lambda_1} < \Autmp(X,\mu)$ is ergodic and its centralizer inside $\Autmp(X,\mu)$ is isomorphic with $G$.

Since $\Autmp(X,\mu)$ is a Polish group, we can choose a countable subgroup $\Lambda$ of $\Autmp(X,\mu)$ whose closure equals the closure of $\widehat{\Lambda_1}$. Since the centralizer of $\Lambda$ equals the centralizer of $\widehat{\Lambda_1}$, it follows that $G$ is a centralizer group in the sense of Definition \ref{def.centralizer-group}.

To prove (ii), take a standard probability space $(Y_0,\eta_0)$ and a subgroup $\Lambda_0 < \Aut(Y_0,\eta_0)$ of nonsingular transformations. Denote by $G$ the centralizer of $\Lambda_0$ inside $\Aut(Y_0,\eta_0)$. We have to prove that $G$ is a centralizer group.

Denote by $\gamma$ the measure on $\R$ given by $d\gamma(t) = \exp(-t) \, dt$. Write $(Y_1,\eta_1) = (Y_0 \times \R, \eta_0 \times \gamma)$. For every $\be \in \Aut(Y_0,\eta_0)$, denote its \emph{Maharam extension} by
$$\betil \in \Autmp(Y_1,\eta_1) : \betil(y,t) = \bigl(\be(y),t + \log \frac{d(\be^{-1}_*(\eta))}{d\eta}(y)\bigr) \; .$$
For every $s \in \R$, denote by $\lambda_s$ the measure scaling automorphism $\lambda_s(y,t) = (y,t+s)$. Then
$$\Psi : \Aut(Y_0,\eta_0) \to \Autmp(Y_1,\eta_1) : \Psi(\be) = \betil$$
is a group homomorphism whose image is equal to
$$\{\zeta \in \Autmp(Y_1,\eta_1) \mid \forall s \in \R : \zeta \circ \lambda_s = \lambda_s \circ \zeta \;\;\text{a.e.}\} \; .$$
The image of $\Psi$ is closed and $\Psi$ is a homeomorphism between $\Aut(Y_0,\eta_0)$ and the image of $\Psi$.

Then define $(Y,\eta) = (Y_1 \times \R , \eta_1 \times \gamma) = (Y_0 \times \R \times \R, \eta_0 \times \gamma \times \gamma)$. For every $s \in \R$, define $\rho_s \in \Autmp(Y,\eta)$ by $\rho_s(y,t_1,t_2) = (y,t_1+s,t_2-s)$. Denote by $\Lambda < \Autmp(Y,\eta)$ the subgroup generated by $\betil \times \id$, $\be \in \Lambda_0$, $\rho_s$, $s \in \R$, and $\id \times \id \times \Autmp(\R,\gamma)$. Denote by $G_1$ the centralizer of $\Lambda$ inside $\Autmp(Y,\eta)$. Since we have already proven (i), it suffices to prove that $G \cong G_1$.

Take $\vphi \in G_1$. Since $\vphi$ commutes with $\id \times \id \times \Autmp(\R,\gamma)$, we get that $\vphi = \psi \times \id$ for some $\psi \in \Autmp(Y_1,\eta_1)$. Since $\vphi$ commutes with all $\rho_s$, we get that $\psi$ commutes with all $\lambda_s$. Since $\psi$ is measure preserving, we get that $\psi = \altil$ for a unique $\al \in G$. So the map $\al \mapsto \Psi(\al) \times \id$ implements the isomorphism $G \cong G_1$.
\end{proof}

\begin{proof}[{Proof of Theorem \ref{thm.main-every-lc}}]
Choose a probability measure $\mu$ that is equivalent with the Haar measure on $G$. View $G < \Aut(G,\mu)$ by right translations. Then the centralizer of $G$ equals the copy of $G$ in $\Aut(G,\mu)$ by left translations. By Theorem \ref{thm.main-centralizer-group.ii}, the group $G$ is a centralizer group. By \cite[Theorem E]{Dep10} (see also \cite[Proposition 8.5]{PV21}), there exists a full II$_1$ factor $M$ with separable predual such that $\Out M \cong G$.
\end{proof}

\section{Permanence properties}\label{sec.permanence}

When $(G_i)_{i \in I}$ is a countable family of Polish groups, the product $\prod_{i \in I} G_i$ with the topology of pointwise convergence is again a Polish group. We start with the following elementary observation.

\begin{proposition}\label{prop.direct-product}
If $(G_i)_{i \in I}$ is a countable family of centralizer groups, the product $\prod_i G_i$ also is a centralizer group.
\end{proposition}
\begin{proof}
Choose standard probability spaces $(X_i,\mu_i)$ and ergodic subgroups $\Lambda_i < \Autmp(X_i,\mu_i)$ whose centralizers are isomorphic with $G_i$ for all $i \in I$. Define the standard probability $(X,\mu) = \prod_{i \in I} (X_i,\mu_i)$. For every $i \in I$, denote by $\Gamma_i < \Autmp(X,\mu)$ the group $\Lambda_i$ acting in the $i$'th coordinate. Define $\Gamma < \Autmp(X,\mu)$ as the subgroup generated by all the $\Gamma_i$, $i \in I$. Since each $\Lambda_i < \Autmp(X_i,\mu_i)$ is ergodic, any element $\al \in \Autmp(X,\mu)$ that centralizes $\Gamma$ must be of the form $\prod_i \al_i$ with $\al_i \in \Autmp(X_i,\mu_i)$. From this the result follows immediately.
\end{proof}

With some more care, we also have the following continuous variant of Proposition \ref{prop.direct-product}.

Recall from \cite[Definition 7.3]{VW24} (see also \cite{Sut85}) the concept of a \emph{measured field} $(G_x)_{x \in X}$ of Polish groups. In a nutshell, this means that $(X,\mu)$ is a standard probability space and that, after discarding from $X$ a Borel set of measure zero, we are given a standard Borel structure on the disjoint union $G = \sqcup_{x \in X} G_x$ such that the following holds: the natural factor map $\pi : G \to X$ is Borel, there exists a dense sequence of Borel sections $\vphi_n : X \to G$ and, writing $G \times_\pi G = \{(g,h) \mid \pi(g) = \pi(h)\}$, the multiplication map $G \times_\pi G \to G$ and the inverse map $G \to G$ are Borel.

Given such a measured field $(G_x)_{x \in X}$ over $(X,\mu)$, we consider the group $\cG$ of all Borel sections $\psi : X \to G$, where we identify sections that are equal a.e. With pointwise multiplication and with the topology of convergence in measure, where $\psi_k \to \psi$ if and only if $\mu\bigl(\{x \in X \mid d(\psi_k(x),\psi(x)) > \eps\}\bigr) \to 0$ for all $\eps > 0$, we get that $\cG$ is a Polish group.

Note that when $X$ is a countable set and $\mu$ is an atomic probability measure that assigns a positive value to each atom, then $\cG = \prod_{i \in X} G_i$ as above. In that sense, $\cG$ is a ``direct integral'' version of the usual product of Polish groups.

\begin{proposition}\label{prop.direct-integral}
Let $(X,\mu)$ be a standard probability space and let $(G_x)_{x \in X}$ be a measured field of second countable locally compact groups. As above, denote by $\cG$ the Polish group of measurable sections of this field, identifying sections that are equal a.e.\ and using the topology of convergence in measure. Then $\cG$ is a centralizer group.
\end{proposition}

\begin{proof}
We discard from $X$ a Borel set of measure zero, so that $G = \sqcup_{x \in X} G_x$ is equipped with a standard Borel structure satisfying the conditions explained above (see \cite[Definition 7.3]{VW24}). Using \cite[Proposition 10.2]{VW24}, we may choose probability measures $\gamma_x$ on $G_x$ such that for every $x \in X$, $\gamma_x$ is equivalent with the Haar measure of $G_x$ and such that for every positive Borel function $F : G \to [0,+\infty)$, the map
$$X \to [0,+\infty] : x \mapsto \int_{G_x} F \, d\gamma_x$$
is Borel. We then denote by $\gamma$ the probability measure on the Borel $\sigma$-algebra of $G$ given by
$$\gamma(\cU) = \int_X \gamma_x(\cU \cap G_x) \, d\mu(x) = \int_X  \Bigl(\int_{G_x}  1_{\cU} \, d\gamma_x  \Bigr) \, d\mu(x) \quad\text{for every Borel set $\cU \subset G$.}$$
By construction, $\pi_*(\gamma) = \mu$. For every Borel section $\vphi : X \to G$, define the Borel bijections
$$\al_\vphi : G \to G : \al_\vphi(g) = \vphi(\pi(g)) g \quad\text{and}\quad \be_\vphi : G \to G: \be_\vphi(g) = g \vphi(\pi(g))^{-1} \; .$$
Since every $\gamma_x$ is equivalent with the Haar measure of $G_x$, it follows that all $\al_\vphi$ and $\be_\vphi$ are nonsingular automorphisms of $(G,\gamma)$. By construction,
$$\al : \cG \to \Aut(G,\gamma) : \vphi \mapsto \al_\vphi \quad\text{and}\quad \be : \cG \to \Aut(G,\gamma) : \vphi \mapsto \be_\vphi$$
are faithful continuous group homomorphisms.

Note that whenever $G_0$ is a locally compact second countable group and $\gamma_0$ is a probability measure on $G_0$ that is equivalent with the Haar measure of $G$, then the embedding $G_0 \hookrightarrow \Aut(G_0,\gamma_0)$ given by left translation has a closed image and is a homeomorphism of $G_0$ onto this image. It follows that the same holds for $\al : \cG \to \Aut(G,\gamma)$.

We claim that $\Delta \in \Aut(G,\gamma)$ belongs to $\al(\cG)$ if and only if $\Delta$ commutes with $\be(\cG)$ and satisfies $\pi(\Delta(g)) = \pi(g)$ for a.e.\ $g \in G$. One implication being obvious, assume that $\Delta \in \Aut(G,\gamma)$ commutes with $\be(\cG)$ and satisfies $\pi(\Delta(g)) = \pi(g)$ for a.e.\ $g \in G$. Choose a dense sequence of Borel sections $\vphi_n : X \to G$. Take a conull Borel set $X_0 \subset X$ such that for every $x \in X_0$, the following holds:
\begin{align*}
& \pi(\Delta(g)) = x \quad\text{for $\gamma_x$-a.e.\ $g \in G_x$, and}\\
& (\Delta \circ \be_{\vphi_n})(g) = (\be_{\vphi_n} \circ \Delta)(g) \quad\text{for all $n \in \N$ and $\gamma_x$-a.e.\ $g \in G_x$.}
\end{align*}
So for every $x \in X_0$, there is a unique $\vphi(x) \in G_x$ such that $\Delta(g) = \vphi(x) g$ for $\gamma_x$-a.e.\ $g \in G_x$. It follows that $\vphi \in \cG$ and $\Delta = \al(\vphi)$. So the claim is proven.

Take any nonatomic standard probability space $(Y_1,\eta_1)$. Choose an injective Borel map $\rho : X \to \Autmp(Y_1,\eta_1)$. Write $(Y,\eta) = (G \times Y_1,\gamma \times \eta_1)$ and define $\delta \in \Autmp(Y,\eta)$ by $\delta(g,y) = (g,\rho(\pi(g))(y))$.

Define $\Lambda < \Aut(Y,\eta)$ as the subgroup generated by $\Lambda_0 \times \Autmp(Y_1,\eta_1)$ and $\delta$. We prove that the centralizer of $\Lambda$ inside $\Aut(Y,\eta)$ is given by $\al(\cG) \times \id$. The proposition then follows from Theorem \ref{thm.main-centralizer-group.ii}.

Choose $\Delta \in \Aut(Y,\eta)$ that commutes with $\Lambda$. Since $\Delta$ commutes with $\id \times  \Autmp(Y_1,\eta_1)$, we get that $\Delta = \Delta_0 \times \id$ for some $\Delta_0 \in \Aut(G,\gamma)$. Since $\Delta$ commutes with $\delta$ and $\rho$ is injective, it follows that $\pi \circ \Delta_0 = \pi$. Since $\Delta_0$ also commutes with $\Lambda_0$, the result follows from the claim above.
\end{proof}

\begin{remark}
It is plausible that for any measured field $(G_x)_{x \in X}$ of Polish groups with the property that a.e.\ $G_x$ is a centralizer group, also the Polish group $\cG$ of measurable sections is a centralizer group. To prove such a result, one would need to make a measurable choice of $\Lambda_x < \Aut(Y_x,\eta_x)$ such that $G_x$ is the centralizer of $\Lambda_x$ inside $\Aut(Y_x,\eta_x)$. While it sounds plausible that such a measurable choice always exists, we were unable to prove this.
\end{remark}

\section{Closed groups of permutations of a countable set}\label{sec.closed-groups-of-permutations}

We denote by $S_\infty$ the Polish group of all permutations of $\N$, equipped with the topology of pointwise convergence. The following result was already proven in \cite[Corollary 8.2]{PV21}, but with the spectral methods of Section \ref{sec.centralizer-groups}, we can give a considerably easier proof.

\begin{theorem}\label{thm.closed-subgroup-Sinfty}
Every closed subgroup of $S_\infty$ is a centralizer group.
\end{theorem}

The main point in proving Theorem \ref{thm.closed-subgroup-Sinfty} is the following lemma, which was obtained as \cite[Lemma 8.6]{PV21} and for which we give an easy spectral proof.

\begin{lemma}\label{lem.mutual-centralizer-Bernoulli}
Let $(X_0,\mu_0)$ be a standard nonatomic probability space. Put $(X,\mu) = (X_0,\mu_0)^\N$. We view $S_\infty < \Aut(X,\mu)$ by permuting the indices and we view $\Autmp(X_0,\mu_0) < \Aut(X,\mu)$ as diagonal transformations.

In this way, $S_\infty$ and $\Autmp(X_0,\mu_0)$ are ergodic subgroups of $\Aut(X,\mu)$ that are each other's centralizer inside $\Aut(X,\mu)$.
\end{lemma}
\begin{proof}
For ever $\al \in \Autmp(X_0,\mu_0)$, we denote by $\vphi_\al \in \Autmp(X,\mu)$ the diagonal transformation $(\vphi_\al(x))_n = \al(x_n)$. When $\al$ is a mixing transformation, $\vphi_\al$ is mixing. So, $\Autmp(X_0,\mu_0)$ is an ergodic subgroup of $\Aut(X,\mu)$. Since $S_\infty < \Autmp(X,\mu)$ contains a Bernoulli shift, also $S_\infty < \Autmp(X,\mu)$ is an ergodic subgroup. To prove the lemma, it thus suffices to prove that $S_\infty$ and $\Autmp(X_0,\mu_0)$ are each other's centralizer inside $\Autmp(X,\mu)$.

First assume that $\vphi \in \Autmp(X,\mu)$ commutes with $S_\infty$. For every $k \in \N$, denote by $S_{\infty,k} < S_\infty$ the stabilizer of $k \in \N$. Since the action of $S_{\infty,k}$ on $(X_0,\mu_0)^{\N \setminus \{k\}}$ is ergodic and since $\vphi$ commutes with all $S_{\infty,k}$, it follows that $\vphi = \prod_{k \in \N} \al_k$, where $\al_k \in \Autmp(X_0,\mu_0)$. Since $\vphi$ commutes with $S_\infty$, it then follows that all $\al_k$ are equal. So, $\vphi = \vphi_\al$ for some $\al \in \Autmp(X_0,\mu_0)$.

Conversely, assume that $\psi \in \Autmp(X,\mu)$ commutes with all $\vphi_\al$, $\al \in \Autmp(X_0,\mu_0)$. For every $\be \in \Autmp(X,\mu)$, we denote by $U_\be \in \cU(L^2(X,\mu))$ the corresponding unitary operator. Below we prove that we may choose a sequence $\al_n \in \Autmp(X_0,\mu_0)$ such that
\begin{equation}\label{eq.my-weak-convergence}
U_{\al_n} \to P_{\C \cdot 1} + \frac{1}{2} P_{(\C \cdot 1)^\perp} \quad\text{weakly in $B(L^2(X_0,\mu_0))$.}
\end{equation}
Write $\vphi_n = \vphi_{\al_n}$. We identity $L^2(X,\mu)$ with the infinite tensor product of the Hilbert spaces $L^2(X_0,\mu_0)$ w.r.t.\ the vector $1 \in L^2(X_0,\mu_0)$. For every finite subset $\cF \subset \N$, we define the closed subspace $H_\cF \subset L^2(X,\mu)$ by
$$H_\cF = \overline{\operatorname{span}} \Bigl( \bigotimes_{i \in \cF} (L^2(X_0,\mu_0) \ominus \C 1) \Bigr) \subset L^2(X,\mu) \; .$$
Note that $H_\cF \perp H_{\cF'}$ when $\cF \neq \cF'$. For every $k \in \N$, we write
$$L_k = \bigoplus_{\cF : |\cF|=k} H_\cF$$
with the convention that $L_0 = \C 1$. It then follows from \eqref{eq.my-weak-convergence} that
$$U_{\vphi_n} \to \sum_{k=0}^\infty 2^{-k} P_{L_k} \quad\text{weakly in $B(L^2(X,\mu))$.}$$
Since $\psi$ commutes with all $\vphi_\al$, we get that $U_\psi$ commutes with all $U_{\vphi_n}$ and hence also with its weak limit. It follows that $U_\psi(L_1) = L_1$.

Choose a Haar unitary $v : X_0 \to \T$ that generates $L^\infty(X_0,\mu_0)$, meaning that $v : X_0 \to \T$ is a measure preserving isomorphism from $(X_0,\mu_0)$ onto the circle with the Lebesgue measure. Denote by $\pi_i : L^2(X_0,\mu_0) \to L^2(X,\mu)$ the embedding as functions only depending on the $i$-th coordinate. Fix $i \in \N$. Define $w : X \to \T$ by $w(x) = v((\psi^{-1}(x))_i)$. So, $w = U_\psi(\pi_i(v))$ and because $U_\psi(L_1) = L_1$, we can uniquely write
$$w = \sum_{j \in \N} \pi_j(v_j) \quad\text{with $v_j \in L^2(X_0,\mu_0) \ominus \C 1$ and $L^2$-convergence.}$$
Since $w^2 = U_\psi(\pi_i(v^2)) \in L_1$, it follows that for all $j \neq k$, $\pi_j(v_j) \pi_k(v_k) = 0$. Taking the $\|\,\cdot\,\|_2$-norm, we get that $\|v_j\|_2 \, \|v_k\|_2 = \|\pi_j(v_j) \pi_k(v_k)\|_2 = 0$ for all $j \neq k$. We conclude that at most one of the $v_j$ is nonzero. So there is a unique $j \in \N$ and $v_j \in L^2(X_0,\mu_0) \ominus \C 1$ such that $U_\psi(\pi_i(v)) = w = \pi_j(v_j)$. Since $v$ generates $L^\infty(X_0,\mu_0)$, we get that $U_\psi(\pi_i(L^\infty(X_0,\mu_0))) \subset \pi_j(L^\infty(X_0,\mu_0))$. Since a similar argument holds for $\psi^{-1}$, we find a permutation $\si \in S_\infty$ and a family $\be_i \in \Autmp(X_0,\mu_0)$ such that $U_\psi(\pi_i(v)) = \pi_{\si(i)}(U_{\be_i}(v))$ for all $i \in \N$ and $v \in L^\infty(X_0,\mu_0)$.

Expressing that $\psi$ commutes with all $\vphi_\al$, it follows that every $\be_i$ commutes with every $\al \in \Autmp(X_0,\mu_0)$. So, $\be_i = \id$ and we have proven that $\psi = \si$.

We still need to construct $\al_n \in \Autmp(X_0,\mu_0)$ satisfying \eqref{eq.my-weak-convergence}. We may realize $(X_0,\mu_0)$ as $(X_1 \times X_2,\mu_1 \times \mu_2)$, where $(X_1,\mu_1) = (\{0,1\}^\Z,\nu^\Z)$, $\nu(0)=\nu(1)=1/2$ and $(X_2,\mu_2)$ is a nonatomic standard probability space. Choose a mixing $\gamma \in \Autmp(X_2,\mu_2)$. Then define $\al_n \in \Autmp(X_1 \times X_2,\mu_1 \times \mu_2)$ by
$$\al_n(x,y) = (x',y') \quad\text{with}\quad \begin{cases} \text{$(x',y') = (x,y)$ if $x_n = 0$,}\\
\text{$x'_k = x_{2n-k}$ for all $k \in \Z$ and $y' = \gamma^n(y)$ if $x_n=1$.}\end{cases}$$
One easily checks \eqref{eq.my-weak-convergence}.
\end{proof}

We now repeat part of the proof of \cite[Proposition 8.11]{PV21} to prove Theorem \ref{thm.closed-subgroup-Sinfty} and use the opportunity to give a more self-contained argument.

\begin{proof}[{Proof of Theorem \ref{thm.closed-subgroup-Sinfty}}]
Fix a closed subgroup $G < S_\infty$. For every integer $n \geq 1$, denote by $I_{n,k} \subset \N^n$ an enumeration of the orbits of the diagonal action $G \actson \N^n$. Denote by $\lambda$ the Lebesgue measure on the circle $\T$ and define
$$(X,\mu) = (\T,\lambda)^\N \times \prod_{n,k} (\T,\lambda)^{I_{n,k}} \; .$$
Define the subgroup $\Lambda_0 < \Autmp(X,\mu)$ consisting of the diagonal transformations $\al^\N \times \id$, $\al \in \Autmp(\T,\lambda)$, acting in the first factor of $(X,\mu)$. Define $\Lambda_{n,k} < \Autmp(X,\mu)$ as the subgroup of all diagonal transformations $\al^{I_{n,k}}$, with $\al \in \Autmp(\T,\lambda)$, that only act in the factor $(\T,\lambda)^{I_{n,k}}$ and act as the identity elsewhere.

For every $x \in X$, we denote by $x_0 \in \T^\N$ and $x_{n,k} \in \T^{I_{n,k}}$ its natural coordinates. For every $n,k$ and $1 \leq j \leq n$, define the transformation $\theta_{n,k,j} \in \Autmp(X,\mu)$ by
\begin{multline*}
\theta_{n,k,j}(x_0,(x_{n',k'})_{n',k'}) = (x_0,(y_{n',k'})_{n',k'}) \\ \text{with}\quad (y_{n',k'})_i = \begin{cases} (x_0)_{i_j} \cdot (x_{n,k})_i &\;\;\text{if $(n',k') = (n,k)$,}\\
(x_{n',k'})_i &\;\;\text{if $(n',k') \neq (n,k)$.}\end{cases}
\end{multline*}
Define $\Lambda < \Autmp(X,\mu)$ as the subgroup generated by $\Lambda_0$, all $\Lambda_{n,k}$ and all $\theta_{n,k,j}$. We prove the centralizer of $\Lambda$ inside $\Autmp(X,\mu)$ is isomorphic with $G$.

Fix a $\psi \in \Autmp(X,\mu)$ that commutes with $\Lambda$. Since $\psi$ commutes with all $\Lambda_0$ and $\Lambda_{n,k}$, by Lemma \ref{lem.mutual-centralizer-Bernoulli}, we find permutations $\si_0$ of $\N$ and $\si_{n,k}$ of $I_{n,k}$ such that $\psi = \si_0 \times \prod_{n,k} \si_{n,k}$. Expressing that $\psi$ commutes with all $\theta_{n,k,j}$, we get that $(\si_{n,k}(i))_j = \si_0(i_j)$ for all $i \in I_{n,k}$ and $1 \leq j \leq n$. This means that $\si_0$ is a permutation of $\N$ with the property that the permutation $\si_0^n$ of $\N^n$ globally preserves all orbits of $G \actson \N^n$, for all $n \geq 1$, and that
$$\psi = \si_0 \times \prod_{n,k} \si_0^n|_{I_{n,k}} \; .$$
To conclude the proof of the theorem, it thus suffices to prove that such a permutation $\si_0$ must belong to $G$.

Take $i_1,\ldots,i_n \in \N$. Since $\si_0^n$ globally preserves the orbits of $G \actson \N^n$, there exists a $g \in G$ such that $\si_0^n(i) = g \cdot i$. This means that $\si_0(i_j) = g \cdot i_j$ for all $j \in \{1,\ldots,n\}$. We have proven that $\si_0$ lies in the closure of $G$ inside $S_\infty$. Since $G$ is a closed subgroup, it follows that $\si_0 \in G$.
\end{proof}

Lemma \ref{lem.mutual-centralizer-Bernoulli} also gives immediately the following result. Given a compact second countable group $K$, recall that the group $\Aut K$ of automorphisms of $K$ is a Polish group as a closed subgroup of the group of homeomorphisms of $K$.

\begin{proposition}\label{prop.Aut-K}
For every compact second countable group $K$, the automorphism group $\Aut K$ is a centralizer group.
\end{proposition}
\begin{proof}
Denote by $\lambda$ the Haar probability measure on $K$. Write $(X,\mu) = (K,\lambda)^\N$. Denote by $\Lambda < \Autmp(X,\mu)$ the subgroup generated by $S_\infty$, acting by permuting the coordinates, and $\theta \in \Autmp(X,\mu)$ given by
$$(\theta(x))_n = \begin{cases} x_1 x_2 &\;\;\text{if $n=1$,}\\
x_n &\;\;\text{if $n \neq 1$.}\end{cases}$$

We prove that the centralizer of $\Lambda$ inside $\Autmp(X,\mu)$ is isomorphic with $\Aut K$. Take $\psi \in \Autmp(X,\mu)$ that commutes with $\Lambda$. For every $\al \in \Autmp(K,\lambda)$ denote by $\vphi_\al \in \Autmp(X,\mu)$ its diagonal product. Since $\psi$ commutes with $S_\infty$ by Lemma \ref{lem.mutual-centralizer-Bernoulli}, we find $\al \in \Autmp(K,\Lambda)$ such that $\psi = \vphi_\al$. Since $\psi$ commutes with $\theta$, we get that $\al(xy) = \al(x) \al(y)$ for a.e.\ $(x,y) \in K \times K$. By \cite[Theorems B.2 and B.3]{Zim84}, there is a unique $\be \in \Aut K$ such that $\al(x) = \be(x)$ for a.e.\ $x \in K$. Since conversely every automorphism of $K$ preserves $\lambda$, we have identified the centralizer of $\Lambda$ with $\Aut K$.
\end{proof}


\begin{thebibliography}{CCM12}\setlength{\itemsep}{-1mm} \setlength{\parsep}{0mm} \small
\bibitem[Dep10]{Dep10} S. Deprez, Explicit examples of equivalence relations and II$_1$ factors with prescribed fundamental group and outer automorphism group. {\it Trans. Amer. Math. Soc.} {\bf 367} (2015), 6837-6876.

\bibitem[FV07]{FV07} S. Falgui\`{e}res and S. Vaes, Every compact group arises as the outer automorphism group of a II$_1$ factor. {\it J. Funct. Anal.} {\bf 254} (2008), 2317-2328.

\bibitem[IPP05]{IPP05} A. Ioana, J. Peterson and S. Popa, Amalgamated free products of weakly rigid factors and calculation of their symmetry groups. {\it Acta Math.} {\bf 200} (2008), 85-153.

\bibitem[JRD20]{JRD20} E. Janvresse, E. Roy and T. De La Rue, Dynamical systems of probabilistic origin: Gaussian and Poisson systems. In {\it Ergodic theory}, Encycl. Complex. Syst. Sci., Springer, New York, 2023. pp.\ 217-232.

\bibitem[PV06]{PV06} S. Popa and S. Vaes, Strong rigidity of generalized Bernoulli actions and computations of their symmetry groups. {\it Adv. Math.} {\bf 217} (2008), 833-872.

\bibitem[PV21]{PV21} S. Popa and S. Vaes, W$^*$-rigidity paradigms for embeddings of II$_1$ factors. {\it Commun. Math. Phys.} {\bf 395} (2022), 907-961.

\bibitem[Roy08]{Roy08} E. Roy, Poisson suspensions and infinite ergodic theory. {\it Ergodic Theory Dynam. Systems} {\bf 29} (2009), 667-683.

\bibitem[Sut85]{Sut85} C.E. Sutherland, A Borel parametrization of Polish groups. {\it Publ. Res. Inst. Math. Sci.} {\bf 21} (1985), 1067-1086.

\bibitem[Vae07]{Vae07} S. Vaes, Explicit computations of all finite index bimodules for a family of II$_1$ factors. {\it Ann. Sci. \'{E}c. Norm. Sup\'{e}r.} {\bf 41} (2008), 743-788.

\bibitem[VW24]{VW24} S. Vaes and L. Wouters, Borel fields and measured fields of Polish spaces, Banach spaces, von Neumann algebras and C$^*$-algebras. {\it Preprint.} \href{https://arxiv.org/abs/2405.16603}{arXiv:2405.16603}

\bibitem[Zim84]{Zim84} R.J. Zimmer, Ergodic theory and semisimple groups. {\it Monogr. Math.} {\bf 81}, Birkh\"{a}user Verlag, Basel, 1984.
\end{thebibliography}
\end{document}